\documentclass[12pt]{amsart}
\usepackage{amsfonts, amssymb, latexsym,tikz,subcaption}

\newtheorem{theorem}{Theorem}[section]
\newtheorem{proposition}[theorem]{Proposition}
\newtheorem{lemma}[theorem]{Lemma}

\newtheorem{claim}[theorem]{Claim}
\newtheorem*{claim*}{Claim}
\newtheorem{corollary}[theorem]{Corollary}
\newtheorem{Main Conjecture}[theorem]{Main Conjecture}

\theoremstyle{remark}
\newtheorem{definition}[theorem]{Definition}
\newtheorem{example}[theorem]{Example}

\theoremstyle{plain}

\newcommand{\cellsize}{19}
\newlength{\cellsz} 
\newsavebox{\cell}
\sbox{\cell}{\begin{picture}(\cellsize,\cellsize)
\put(0,0){\line(1,0){\cellsize}}
\put(0,0){\line(0,1){\cellsize}}
\put(\cellsize,0){\line(0,1){\cellsize}}
\put(0,\cellsize){\line(1,0){\cellsize}}
\end{picture}}
\newcommand\cellify[1]{\def\thearg{#1}\def\nothing{}%
\ifx\thearg\nothing
\vrule width0pt height\cellsz depth0pt\else
\hbox to 0pt{\usebox{\cell} \hss}\fi%
\vbox to \cellsz{
\vss
\hbox to \cellsz{\hss$#1$\hss}
\vss}}
\newcommand\tableau[1]{\vtop{\let\\\cr
\baselineskip -16000pt \lineskiplimit 16000pt \lineskip 0pt
\ialign{&\cellify{##}\cr#1\crcr}}}

\newcommand{\excise}[1]{}

\begin{document}
\pagestyle{plain}

\title{Minimal equations for matrix Schubert varieties}
\author{Shiliang Gao}
\author{Alexander Yong}
\address{Dept.~of Mathematics, University of Illinois at Urbana-Champaign, Urbana, IL 61801}

\email{sgao23@illinois.edu, ayong@illinois.edu}

\begin{abstract}
Explicit minimal generators for Fulton's Schubert determinantal ideals are determined along
with some implications.
\end{abstract}



\date{November 19, 2023}
\maketitle

\section{Introduction and Results}

Let ${\sf Mat}_{n\times n}$ be the space of $n\times n$ matrices over a field $\Bbbk$; the coordinate ring is 
$R={\Bbbk}[x_{ij}]_{1\leq i,j\leq n}$. Let $GL_n$ be the general linear group of invertible $n\times n$ matrices with a 
Borel subgroup $B$ of upper triangular matrices and $B_-$ of lower triangular matrices. Let $B_-\times B$ act on 
${\sf Mat}_{n\times n}$ by $(b_-,b)\cdot M= b_{-}Mb^{-1}$. Let $w$ be a permutation in the symmetric group ${\mathfrak S}_n$ on 
$[n]=\{1,2,\ldots,n\}$, and suppose $M_w$ is its permutation matrix with $1$ in position $(i,w(i))$ and $0$'s elsewhere. 

\begin{definition}[\cite{Fulton,KM:annals}]
The \emph{matrix Schubert variety} $X_w$ is the $B_{-}\times B$ orbit closure of $M_w$.
\end{definition}

\begin{definition}[\cite{Fulton,KM:annals}]
The \emph{Schubert determinantal ideal} $I_w\subset R$ is the defining ideal of $X_w$.
\end{definition}

Since \cite{Fulton}, there has been interest in matrix Schubert varieties and the Schubert determinantal ideals; see, \emph{e.g.,} \cite{KM:annals, KMY, Hsiao, Fink, Hamaker, Klein, Klein.Weigandt, Rajchgot, Pechenik} and references therein. 

Let $r_{i,j}=r_{i,j}(w)$ be the \emph{rank function} of $w$. It counts the number of $1$'s weakly northwest of position $(i,j)$ in $M_w$. Let $M^{[i,j]}$ denote the
northwest $i\times j$ submatrix of a generic matrix $M\in {\sf Mat}_{n\times n}$. In \cite{Fulton}, it is proved that
$I_w$ is indeed generated by determinants:
\begin{equation}
\label{eqn:generators}
I_w=\langle (r_{i,j}+1)\times (r_{i,j}+1) \text{ size minors of $M^{[i,j]}$}, (i,j)\in [n]^2\rangle,
\end{equation}
and that this ideal is prime \cite[Corollary~3.13]{Fulton}.

In \emph{loc.~cit.}, W.~Fulton minimizes the \emph{description} of the generators (\ref{eqn:generators}).
The \emph{Rothe diagram} of $w$ is
\[D(w)=\{(i,j)\in [n]^2: j<w(i), i<w^{-1}(j)\}.\]
\emph{Fulton's essential set} is
\[E(w)=\{(i,j)\in D(w): (i+1,j),(i,j+1)\not\in D(w)\}.\]
W.~Fulton proved that $I_w=\langle (r_{i,j}+1)\times (r_{i,j}+1) \text{\ size minors of $M^{[i,j]}$}, (i,j)\in E(w)\rangle$.
This is a minimal list of rank conditions needed to describe $I_w$ but
does not provide a minimal set of generators. The minors in this description are called the \emph{essential minors} of $I_w$.

\begin{example}[Essential minors do not form a minimal generating set] The reader can check
\[I_{3142}=\left\langle x_{11}, x_{12},
\left|\begin{matrix} x_{21} & x_{22}\\ x_{31} & x_{32}\end{matrix}\right|,
\left|\begin{matrix} x_{11} & x_{12}\\ x_{21} & x_{22}\end{matrix}\right|,
\left|\begin{matrix} x_{11} & x_{12}\\ x_{31} & x_{32}\end{matrix}\right|\right\rangle.\]
The latter two essential minors can be dispensed with; they are implied by the first two.
\end{example}

For $I,J\subseteq [n]$ with $|I| = |J|$, define $m_{I,J}$ 
to be the determinant of the submatrix of $M$ with row and column indices $I$ and $J$ respectively. 
An essential generator $m_{I,J}$ 
\emph{belongs to} $(i,j)\in E(w)$ if $I\subseteq [i],J\subseteq[j]$ and $r = |I| = |J| = r_{i,j}+1$.
\begin{definition}
A minor $m_{I,J}$
\emph{attends} $M^{[i',j']}$ if $|I\cap[i']|> r_{i',j'}$ and $|J\cap[j']|= r_{i,j}+1$ or
$|I\cap[i']|=r_{i,j}+1$ and $|J\cap[j']|> r_{i',j'}$. 
\end{definition}

\begin{definition}
A minor $m_{I,J}$ that belongs to $(i,j)\in E(w)$ is \emph{elusive} if it does not attend $M^{[i',j']}$ for all
$(i',j')\in E(w)$ such that $r_{i',j'}<r_{i,j}$.
\end{definition}

\begin{theorem}
\label{thm:main}
$I_w$ is minimally generated by elusive minors. Moreover, for any $b\in D(w)$
there exists an elusive minor whose southeast corner is $b$.
\end{theorem}

\begin{example}\label{exa:enumerate}
Let $w=619723458$. An example of elusive minor is $m_{\{1,2,3\},\{5,7,8\}}$, whereas $m_{\{1,2,3\},\{4,5,8\}}$ is not elusive since it attends $M^{[4,5]}$. 

Theorem~\ref{thm:main} is a handy way to hand compute the size of a minimal
generating set. Here, the minimal generating set
contains $5$ generators of degree $1$, ${3 \choose 2}{5\choose 2}$ generators of degree $2$ and $1+{5\choose 1}{3 \choose 2}$
generators of degree $3$. All $5$ degree $1$ essential minors are elusive, a degree $2$ essential minor is elusive if and only if $1\notin I$, and a degree $3$ essential minor is elusive if and only if $|J\cap[5]|\leq 1$.

\begin{figure}[h!]
\centering
\begin{tikzpicture}[scale = 0.5]
    \fill [gray, opacity  = 0.25] (0,9) rectangle (5,8);
    \fill [gray, opacity  = 0.25] (1,7) rectangle (5,5);
    \fill [gray, opacity  = 0.25] (6,7) rectangle (8,6);
    \draw (0,0)--(9,0)--(9,9)--(0,9)--(0,0);
    \draw (1,0) -- (1,9);
    \draw (2,0) -- (2,9);
    \draw (3,0) -- (3,9);
    \draw (4,0) -- (4,9);
    \draw (5,0) -- (5,9);
    \draw (6,0) -- (6,9);
    \draw (7,0) -- (7,9);
    \draw (8,0) -- (8,9);
    \draw (0,1) -- (9,1);
    \draw (0,2) -- (9,2);
    \draw (0,3) -- (9,3);
    \draw (0,4) -- (9,4);
    \draw (0,5) -- (9,5);
    \draw (0,6) -- (9,6);
    \draw (0,7) -- (9,7);
    \draw (0,8) -- (9,8);

    \node at (5.5,8.5) {$\bullet$};
    \node at (0.5,7.5) {$\bullet$};
    \node at (8.5,6.5) {$\bullet$};
    \node at (6.5,5.5) {$\bullet$};
    \node at (1.5,4.5) {$\bullet$};
    \node at (2.5,3.5) {$\bullet$};
    \node at (3.5,2.5) {$\bullet$};
    \node at (4.5,1.5) {$\bullet$};
    \node at (7.5,0.5) {$\bullet$};
    \draw[line width = 0.35mm] (5.5,0) -- (5.5,8.5) -- (9,8.5);
    \draw[line width = 0.35mm] (0.5,0) -- (0.5,7.5) -- (9,7.5);
    \draw[line width = 0.35mm] (8.5,0) -- (8.5,6.5) -- (9,6.5);
    \draw[line width = 0.35mm] (6.5,0) -- (6.5,5.5) -- (9,5.5);
    \draw[line width = 0.35mm] (1.5,0) -- (1.5,4.5) -- (9,4.5);
    \draw[line width = 0.35mm] (2.5,0) -- (2.5,3.5) -- (9,3.5);
    \draw[line width = 0.35mm] (3.5,0) -- (3.5,2.5) -- (9,2.5);
    \draw[line width = 0.35mm] (4.5,0) -- (4.5,1.5) -- (9,1.5);
    \draw[line width = 0.35mm] (7.5,0) -- (7.5,0.5) -- (9,0.5);
    \end{tikzpicture}
    \caption{Rothe diagram of $w=619723458$, the boxes of $D(w)$ are shaded.}
\end{figure}
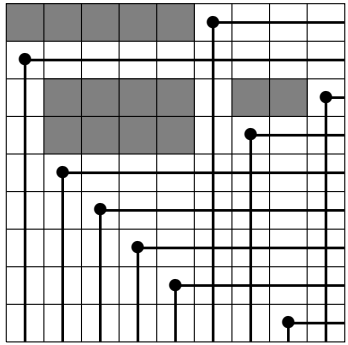

An exercise is $w=13865742$ \cite[Example~1.3.5]{KM:annals}. A minimum generating set is of size $104$ consists of $21$ many $2\times 2$ minors and $83$
many $3\times 3$ minors.\footnote{This example is also considered in the unpublished Section~3 of earlier (v1, v2) {\sf arXiv} preprint versions of \cite{KM:annals}. The notion of \emph{attends} is more general than ``\emph{causes}'' used in those preprints and the published version.} 
\end{example}

A.~Knutson-E.~Miller \cite[Theorem~B]{KM:annals} proves that the essential minors of $I_w$ form a Gr\"obner basis with respect to any \emph{antidiagonal term order} $\prec$, that
is, a monomial order that picks the antidiagonal term of an minor as the initial term (one example is the lexicographic ordering obtained by reading the rows of the generic matrix right to left in rows, and from top to bottom). 
A.~Knutson, E.~Miller and the second author \cite[Theorem~3.8]{KMY} prove the same result for any \emph{diagonal term order}
$\prec'$ (where the order picks the diagonal term of an minor as the initial term), but under the hypothesis that $w$ is \emph{vexillary} (that is $2143$-avoiding; see the definition of pattern
avoidance below). We refine these statements:

\begin{corollary}
\label{cor:Grobner}
The set of elusive minors is a Gr\"obner basis for $I_w$ under an antidiagonal term order $\prec$. 
If we assume $w$ is vexillary, the same statement is true under diagonal term order $\prec'$.
That is, in either case, $I_w$ has a Gr\"obner basis given by a set of
minimal generators.
\end{corollary}

The codimension of $X_w\subseteq \Bbbk^{n^2}$ is $\ell(w)=\#D(w)$, that is, the number of inversions of $w$
\cite[Corollary~3.13]{Fulton}. Since the size of a minimal generating set is an invariant, $X_w$
is  a complete intersection if and only if the size of the set of its elusive minors is $\ell(w)$. Using this, we give a self-contained proof of the result below of H.~Ulfarsson-A.~Woo \cite[Corollary~6.3]{Woo}, which is a pattern avoidance characterization of matrix Schubert varieties that are complete intersections. Their result came after an earlier characterization by J.~C.~Hsiao \cite[Theorem~5.2]{Hsiao} which depends on the Gr\"obner basis theorem of \cite[Theorem~B]{KM:annals}.

Recall $w\in {\mathfrak S}_n$ \emph{pattern includes} $u\in S_{m}$ if there exist indices $i_1<i_2<\ldots<i_m$
such that $w(i_1),\ldots, w(i_m)$ is in the same relative order as $u(1),\ldots,u(m)$. Furthermore, $w$ avoids
$u$ if no such indices exist.

\begin{corollary}[\cite{Hsiao, Woo}]\label{cor:lci}
$X_w$ is a complete intersection if and only if $w$ avoids $1342, 1432, 1423$.\footnote{In \cite[Corollary~6.3]{Woo}, the additional patterns $31524,24153$ and $351624$ are listed. However these follow from the size $4$ pattterns.}
\end{corollary}

In fact, the proofs of \cite{Hsiao, Woo} construct a minimal set of generators for $I_w$ to prove ``$\Leftarrow$''. However, their arguments do not do this outside of those cases.

\section{Proofs}

\subsection{Proof of Theorem~\ref{thm:main}}
Suppose a minor $m$ belonging to $(i,j)\in E(w)$ is not elusive; say it attends $M^{[i',j']}$ for $(i',j')\in E(w)$ satisfying $r_{i',j'}(w)<r_{i,j}(w)$. Then it follows by induction using cofactor expansion that $m$ is in the ideal
generated by the $(r_{i',j'}+1)\times (r_{i',j'}+1)$ minors belonging to $M^{[i',j']}$. Hence $m$ can be dispensed with.

Conversely, suppose $m=m_{I,J}$  belonging to $e=(i,j)\in E(w)$ is elusive.
Let $I=\{1\leq i_1<i_{2}<\ldots<i_{r+1}\leq i\}$ and $J=\{1\leq j_1<j_2<\ldots<j_{r+1}\leq j\}$ where $r=r_{i,j}(w)$. 
\begin{claim}\label{claim:rank}
    For any $1\leq k<r+1$, we have $r_{i_k,j}\geq k$ and $r_{i,j_k}\geq k$.
\end{claim}
\noindent \emph{Proof of Claim~\ref{claim:rank}:} We will prove $r_{i_k,j}\geq k$; $r_{i,j_k}\geq k$ follows from the same reasoning. 
We proceed by induction. For $k=1$, suppose $r_{i_1,j} = 0 <k$, then $(i_1,j)\in D(w)$. There is $(a,b)\in E(w)$ 
weakly southeast of $(i_1,j)$ and in the same connected component of $D(w)$. Since $r_{a,b} = r_{i_1,j} = 0$, $b\geq j$ and $I\cap[a]\geq 1>0$, $m_{I,J}$ attends $M^{[a,b]}$, a contradiction. 

Now suppose $r_{i_k,j}\geq k$ for all $1\leq k<s$ for some $s<r+1$. If $(i_{s},j)\in D(w)$, there is a $(a,b)\in E(w)$ that is in the same connected component as $(i_{s},j)$ and weakly southeast of $(i_{s},j)$. If $r_{i_{s},j} = s-1$, then $m_{I,J}$ attends $M^{[a,b]}$, a contradiction. So $r_{i_{s},j}\geq s$ in this case. Now if $(i_{s},j)\notin D(w)$, since $(i,j)\in D(w)$ and $i>i_{s}$, we know that $w(i_{s}) < j$. Since $r_{{i_{s-1}},j}\geq s-1$, we see $r_{i_{s},j}\geq r_{i_{s-1},j}+1 = s$, completing the induction step.
\qed

 To prove that $m=m_{I,J}$ is necessary as a generator, it suffices
to find a point $P\in {\sf Mat}_{n\times n}$ such that $m$ does not vanish at $P$ but every other essential generator does vanish.
Set
$P_{a,b}=1$ if $a=i_t$ and $b=j_{r-t+2}$ for $1\leq t\leq r+1$, and let all other entries be $0$. In words, $P$ places $1$'s on the antidiagonal of $m$. Evidently $m$ does not vanish at $P$.

It remains to prove all other essential minors do vanish at $P$. Suppose, to the contrary that 
$m'$ is a minor that belongs to $e'\in E(w)$ but does not vanish at $P$. Since the only minor of size at least $r_e(w)+1$ that does not vanish at $P$ is $m$, $r_{e'}(w)<r_e(w)$. Let $e' = (i',j')$. If $e'$ is not in the rectangle with corners $(1,1)$ and $(i_{r+1}-1,j_{r+1}-1)$,
by definition, $m$ attends $M^{[e']}$ contradicting the assumption that $m$ is elusive. 
Thus, the only possibility is that $i'<i_{r+1}$ and $j'<j_{r+1}$ as depicted in Figure~\ref{fig:Jan9aaa}.
\begin{figure}[h]
    \begin{center}
    
        \begin{tikzpicture}[scale=0.45]
        \fill[color=black!40, fill=yellow, very thick] (7,8) rectangle (6,9);
        \fill[color=black!40, fill=yellow, very thick] (4,6) rectangle (3,7);
        \fill[color=black!40, fill=yellow, very thick] (2,3) rectangle (1,4);
        \draw[black, thick] (0,0) -- (10,0) -- (10,10) -- (0,10) -- (0,0);
        \draw[black, thick] (7,1) -- (8,1) -- (8,2) -- (7,2) -- (7,1);
        \draw[black, thick] (5,4) -- (6,4) -- (6,5) -- (5,5) -- (5,4);
        \draw[red, thick, dashed] (0,8) -- (8,8);
        \draw[red, thick, dashed] (2,10) -- (2,1);
        \draw[blue, thick] (0,6) -- (8,6);
        \draw[blue, thick] (4,10) -- (4,1);
        \draw (7.5,1.5) node{$e$};
        \draw (5.5,4.5) node{$e'$};
        \draw (-0.35, 1.5) node{$i$};
        \draw (-0.35, 3.5) node{$i_3$};
        \draw (-0.35, 6.5) node{$i_2$};
        \draw (-0.35, 8.5) node{$i_1$};
        \draw (-0.35, 4.5) node{$i'$};
        \draw (7.5, 10.45) node{$j$};
        \draw (6.5, 10.45) node{$j_3$};
        \draw (3.5, 10.45) node{$j_2$};
        \draw (1.5, 10.45) node{$j_1$};
        \draw (5.5, 10.45) node{$j'$};
        \draw (6.5, 8.5) node{$1$};
        \draw (3.5, 6.5) node{$1$};
        \draw (1.5, 3.5) node{$1$};
        \draw[gray, dashed] (0,1) -- (7,1);
        \draw[gray, dashed] (8,10) -- (8,2);
        \end{tikzpicture}
    \caption{Relative position of $e,e'$ and $P$}
    \label{fig:Jan9aaa}    
    \end{center}
\end{figure}
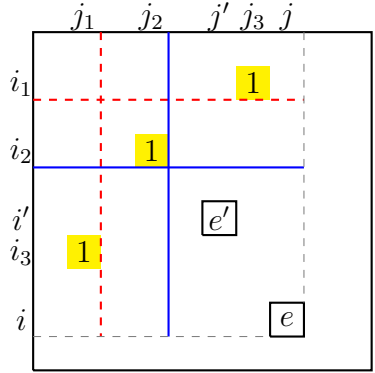

Let us assume that $i_p\leq i'<i_{p+1}$ and $j_{\ell}\leq j'<j_{\ell+1}$ for some $0\leq p,\ell<r+1$, where $i_0 = j_0: = 1$
(in the figure, $p=\ell=2$). Since $m'$ does not vanish at $P$, it is straightforward that $p+\ell>r+1$ 
(otherwise, $m'$ only involves $0$ entries)
and $r_{i',j'}<p+\ell-(r+1)$ (the right hand is the number of $1$'s that appear in the northwest $i_p\times j_{\ell}$
rectangle of $M_w$).
In particular, this implies that 
\[r_{i_p,j_\ell}<p+\ell-(r+1).\]
By Claim~\ref{claim:rank}, $r_{i,j_\ell}\geq \ell$. Since $r_{i,j} = r$, we obtain
\[r_{i_p,j}-r_{i_p,j_\ell}\leq r_{i,j}-r_{i,j_\ell} \leq r-\ell.\]
Hence,
\[r_{i_p,j}<p+\ell-(r+1)+r-\ell = p-1.\]
Yet by Claim~\ref{claim:rank}, $r_{i_p,j}\geq p$, a contradiction. So all other essential minors vanish on $P$.

We now turn to the second statement of the theorem, restated here in more exact form:

\begin{proposition}\label{prop:elusiveSE}
    For $b \!=\! (i,j)\!\in\! D(w)$, $m_{[i-r,i],[j-r,j]}$ is an elusive minor with southeast corner $b$.
\end{proposition}
\begin{proof}
    Let $r = r_b(w), I = [i-r,i]$ and $J = [j-r,j]$. We will show that $m_{I,J}$ is an elusive minor. Since $r_e(w) = r$ for any $e\in E(w)$ that is in the same connected component as $b$ in $D(w)$, $m_{I,J}$ belongs to any such $e$. Suppose $m_{I,J}$ attends $M^{[e']}$ for some $e' = (i',j')\in E(w)$. We can assume, without loss, that $i'< i$ 
    and $j'>j$.
 
    Since $b,e'\in D(w)$, we know that $(i',j)\in D(w)$. Let $k = i-i'$, we then have
    \[r_{i',j}(w)\geq r-k+1.\]
    Since $r_{e'}(w)\geq r_{i',j}(w)$,
    \[r_{e'}(w)\geq r-k+1.\]
    Thus any minor that belongs to $e'$ has size at least $r-k+2$. Since 
    \[|I\cap[i']| = r-k+1 <r-k+2,\]
    $m_{I,J}$ does not attend $M^{[e']}$, a contradiction. Therefore $m_{I,J}$ is elusive, as claimed.
\end{proof}

\subsection{Proof of Corollary~\ref{cor:Grobner}}
For $f\in R$, let ${\sf init}_{\prec}(f)$ be the initial term of $f$ under $\prec$. By definition, a generating set ${\mathcal S}$ of $I_w$ is a Gr\"obner basis if 
\[\langle {\sf init}_{\prec}(f): f\in {\mathcal S}\rangle=
\langle {\sf init}_{\prec}(f): f\in I_w\rangle.\]
By \cite[Theorem~B]{KM:annals}, the essential minors are a Gr\"{o}bner basis under $\prec$.
Therefore it suffices to show that if $m$ is a non-elusive minor then
${\sf init}_{\prec} \, m$ is divisible by ${\sf init}_{\prec} \, m'$ where $m'$ is an elusive minor. We proceed by induction, ordering the essential set by rank value. In the base case where the rank is $0$, all the associated
$1\times 1$ minors are elusive, trivially. Suppose $m$ belongs to $(i,j)\in E(w)$ but is not elusive. 
Since $m=m_{I,J}$ is not elusive we may suppose, without loss of generality, that there is 
$(i',j')\in E(w)$ with $r_{i',j'}(w)<r_{i,j}(w)$
such that $|I\cap[i']|> r_{i',j'}$ and $|J\cap[j']|= r_{i,j}+1$ (the argument in the other case is similar).
Thus, there is a minor $m_{I',J'}$ that belongs to $(i',j')$ (and of size $(r_{i',j'}+1)\times (r_{i',j'}+1)$) whose
antidiagonal term divides that of $m_{I,J}$. More precisely $I'$ consists of the $r_{i',j'}+1$ smallest indices of $I$, and $J'$ consists of the $r_{i',j'}+1$ largest indices of $J$. If $m_{I',J'}$ is elusive we are done. Otherwise by induction
${\sf init}_{\prec} \, m_{I',J'}$ is divisible by ${\sf init}_{\prec} \, m_{I'',J''}$ for some elusive $m_{I'',J''}$ in which case
${\sf init}_{\prec} m_{I'',J''}$ divides ${\sf init}_{\prec} m_{I,J}$, as desired. 

If $w$ is vexillary and
$e,e'\in E(w)$ then $e$ cannot be strictly northwest of $e'$. Using this and the Gr\"obner basis result \cite[Theorem~3.8]{KMY}, the second sentence follows like the first.

The final sentence of the statement then is immediate from Theorem~\ref{thm:main}.

\subsection{Proof of Corollary~\ref{cor:lci}}  

\begin{lemma}
\label{lemma:firstone}
$X_w$ is a complete intersection if and only if $X_{w^{-1}}$ is a complete intersection.
\end{lemma}
\begin{proof}
Notice that $D(w)$ is the transpose of $D(w^{-1})$. The Lemma then follows from (\ref{eqn:generators}).
\end{proof}
\begin{lemma}[Shifting]\label{shifting}
If $m_{I,J}$ is an elusive minor of $I_w$ then $m_{I',J'}$ is elusive whenever $I'=(I-\{i\})\cup \{t\}$ where $t>i$ and
$t\not\in I$ or similarly for $J'$.
\end{lemma}
\begin{proof}
Since $|I'\cap [i']| \leq |I\cap [i']|, |J'\cap [j']|\leq |J\cap [j']|$ for all $i',j'\in [n]$, the lemma is immediate from the definitions.
\end{proof}
$(\Rightarrow)$: We prove the contrapositive. Suppose $w$ pattern embeds $1342$ or $1432$. Let $a_1<a_2<a_3<a_4$ be such that $w(a_1)<w(a_4)<w(a_2),w(a_3)$. Set $b_1 = w(a_1), b_2 = w(a_4), b_3 = w(a_2), b_4 = w(a_3)$, then $e = (a_3,b_2)\in D(w)$ as shown below in Figure~\ref{fig:third}.
\begin{figure}[h]
    \begin{center}
    \begin{subfigure}{0.45\textwidth}
    \hspace{1cm}
    \begin{tikzpicture}[scale = 0.40]
       \draw[black, thick] (0,0) -- (10,0) -- (10,10) -- (0,10) -- (0,0);
        \draw[black, thick] (3,6) -- (3,7) -- (4,7) -- (4,6) -- (3,6) ;
        \draw[black, thick] (3,4) -- (4,4) -- (4,5) -- (3,5) -- (3,4) ;
        \draw[black, thick] (5.5,6.5) -- (5.5,0);
        \draw[black, thick] (5.5,6.5) -- (10,6.5);
        \draw[black, thick] (1.5,8.5) -- (1.5,0);
        \draw[black, thick] (1.5,8.5) -- (10,8.5);
        \draw[black, thick] (3.5,1.5) -- (3.5,0);
        \draw[black, thick] (3.5,1.5) -- (10,1.5);
        \draw[black, thick] (7.5,4.5) -- (7.5,0);
        \draw[black, thick] (7.5,4.5) -- (10,4.5);
        \draw (3.5,4.5) node{$e$};
        \draw (1.5,8.5) node{$\bullet$};
        \draw (5.5,6.5) node{$\bullet$};
        \draw (7.5,4.5) node{$\bullet$};
        \draw (3.5,1.5) node{$\bullet$};
        \draw (-0.55, 1.5) node{$a_4$};
        \draw (-0.55, 6.5) node{$a_2$};
        \draw (-0.55, 8.5) node{$a_1$};
        \draw (-0.55, 4.5) node{$a_3$};
        \draw (7.5, 10.55) node{$b_4$};
        \draw (3.5, 10.55) node{$b_2$};
        \draw (1.5, 10.55) node{$b_1$};
        \draw (5.5, 10.55) node{$b_3$};
        
    \end{tikzpicture}
    \caption{\normalsize{$w$ embeds $1342$}}
    \label{fig:1342}
    \end{subfigure}
    \begin{subfigure}{0.45\textwidth}
    \hspace{1cm}
        \begin{tikzpicture}[scale=0.40]
        \draw[black, thick] (0,0) -- (10,0) -- (10,10) -- (0,10) -- (0,0);
        \draw[black, thick] (3,6) -- (3,7) -- (4,7) -- (4,6) -- (3,6) ;
        \draw[black, thick] (3,4) -- (4,4) -- (4,5) -- (3,5) -- (3,4) ;
        \draw[black, thick] (7.5,6.5) -- (7.5,0);
        \draw[black, thick] (7.5,6.5) -- (10,6.5);
        \draw[black, thick] (1.5,8.5) -- (1.5,0);
        \draw[black, thick] (1.5,8.5) -- (10,8.5);
        \draw[black, thick] (3.5,1.5) -- (3.5,0);
        \draw[black, thick] (3.5,1.5) -- (10,1.5);
        \draw[black, thick] (5.5,4.5) -- (5.5,0);
        \draw[black, thick] (5.5,4.5) -- (10,4.5);
        \draw (3.5,4.5) node{$e$};
        \draw (1.5,8.5) node{$\bullet$};
        \draw (7.5,6.5) node{$\bullet$};
        \draw (5.5,4.5) node{$\bullet$};
        \draw (3.5,1.5) node{$\bullet$};
        \draw (-0.55, 1.5) node{$a_4$};
        \draw (-0.55, 6.5) node{$a_2$};
        \draw (-0.55, 8.5) node{$a_1$};
        \draw (-0.55, 4.5) node{$a_3$};
        \draw (7.5, 10.55) node{$b_4$};
        \draw (3.5, 10.55) node{$b_2$};
        \draw (1.5, 10.55) node{$b_1$};
        \draw (5.5, 10.55) node{$b_3$};
        \end{tikzpicture}
        \caption{\normalsize{$w$ embeds $1432$}}
        \label{fig:1432}
    \end{subfigure}
    \caption{\label{fig:third}}
    \label{}    
    \end{center}
\end{figure}

Furthermore, we can assume, without loss, that for any $a<a_1$, we have $w(a)>b_2$. That is, we can pick 
$(a_1,b_1)$ to be the highest non-zero entry in $M_w$ that is strictly northwest of $e$. Since $w(a_2) = b_3>b_2$, we have $a_3>r_e(w)+1$. Set $r = r_e(w)$ and let 
\[I = [a_3-r,a_3], I' = \{a_3-r-1\}\cup[a_3-r+1,a_3], J = [b_2-r,b_2].\]
Since $m_{I,J}$ is elusive (by Proposition~\ref{prop:elusiveSE}), to show that $m_{I',J}$ is also elusive, it is enough to prove that there does not exist $(i',j')\in E(w)$ such that $J\cap[j'] = J$, $r_{i',j'} = 0$ and $I\cap[i'] = \{a_3-r-1\}$. 
Suppose not. Indeed, since $w(a_2)>b_2$, we know that $a_3-r-1\geq a_1$ and thus $r_{i',j'}\geq 1$,
a contradiction. Therefore, there are at least two elusive minors whose southeast corner is $e$ and therefore, by Theorem~\ref{thm:main}, 
$X_w$ is not a complete intersection. 

Since $w$ includes $1342$ if and only if $w^{-1}$ includes $1423$, we are done by Lemma~\ref{lemma:firstone}.

$(\Leftarrow)$: We again argue the contrapositive. Suppose $X_w$ is not a complete intersection. By Theorem~\ref{thm:main} there is either 
$e = (i,j)\in D(w)$ having more than one elusive minor $m$ with southeast corner $e$, or there is an $e=(i,j)\in [n]^2-D(w)$
that is the southeast corner of an elusive minor $m'$ that belongs to $e'\in E(w)$. In the second case, by
using Lemma~\ref{shifting} to repeatedly shift the southmost row and/or eastmost column used by
$m'$ one obtains another elusive minor $m''$ with southeast corner $e'$ which
is different than the elusive minor from the proof of Proposition~\ref{prop:elusiveSE}. Hence we assume we are in the first
case.

Let $r = r_e(w)$. By using Lemma~\ref{lemma:firstone} or Lemma~\ref{shifting}, we may assume that $m_{I',J}$ is an elusive minor where
\[I' = \{i-r-1\}\cup[i-r+1,i], J = [j-r,j].\]
Since $m_{I',J}$ is elusive, $r_{i-r-1,j}\geq 1$. Also, since $(i,j)\in D(w)$, wither $w(i-r-1)<j$ or $(i-r-1,j)\in D(w)$ is true.

Suppose $w(i-r-1)<j$, by the pigeonhole principle, there exists $a$ such that $(a,j)\in D(w)$ and $i-r-1<a<i$. As a result, 
\[i-r-1<a<i<w^{-1}(j)\ \text{and}\  w(i-r-1)<j<w(a),w(i).\]
Therefore $w$ embeds $1342$ or $1432$, and we are done.

Hence $(i-r-1,j)\in D(w)$. Since $r_{i-r-1,j}\geq 1$, there exists $a<i-r-1$ such that $w(a)<j$. Since $(i-r-1,j)\in D(w)$, we get $w(i-r-1)>j$. We then have
\[a<i-r-1<i<w^{-1}(j)\ \text{and }w(a)<j<w(i-r-1),w(i).\]
Therefore $w$ embeds $1342$ or $1432$. \qed

\section*{Acknowledgements}
We thank the anonymous referee for their helpful and detailed comments. 
The authors were supported by an NSF RTG in Combinatorics. Additionally, SG was supported by an NSF graduate fellowship and AY by a Simons Collaboration Grant.


\begin{thebibliography}{99}

\bibitem{Fink}
Fink, Alex; Rajchgot, Jenna; Sullivant, Seth. Matrix Schubert varieties and Gaussian conditional independence models. J. Algebraic Combin. 44 (2016), no. 4, 1009--1046.

\bibitem{Fulton}
Fulton, William. \emph{Flags, Schubert polynomials, degeneracy loci, and determinantal formulas.} Duke Math. J. 65 (1992), no. 3, 381--420.

\bibitem{Hamaker}
Hamaker, Zachary; Pechenik, Oliver; Weigandt, Anna. Gr\"obner geometry of Schubert polynomials through ice. Adv. Math. 398 (2022), Paper No. 108228, 29 pp.

\bibitem{Hsiao}
Hsiao, Jen-Chieh. \emph{On the $F$-rationality and cohomological properties of matrix Schubert varieties.} 
Illinois J. Math. 57 (2013), no. 1, 1--15. 

\bibitem{Klein}
Klein, Patricia. \emph{Diagonal degenerations of matrix Schubert varieties}, preprint, 2021.
\textsf{arXiv:2008.01717}

\bibitem{Klein.Weigandt}
Klein, Patricia; Weigandt, Anna. \emph{Bumpless pipe dreams encode GrÃ¶bner geometry of Schubert polynomials},
preprint, 2021. \textsf{arXiv:2108.08370}

\bibitem{KM:annals} 
Knutson, Allen; Miller, Ezra. \emph{Gr\"obner geometry of Schubert polynomials.} Ann. of Math. (2) 161 (2005), no. 3, 1245--1318. 

\bibitem{KMY}
Knutson, Allen; Miller, Ezra; Yong, Alexander. \emph{Gr\"obner geometry of vertex decompositions and of flagged tableaux.} J. Reine Angew. Math. 630 (2009), 1--31. 

\bibitem{Pechenik}
Pechenik, Oliver; Speyer, David; Weigandt, Anna. \emph{Castelnuovo-Mumford regularity of matrix Schubert varieties}, preprint, 2021. \textsf{arXiv:2111.10681}

\bibitem{Rajchgot}
Rajchgot, Jenna; Ren, Yi; Robichaux, Colleen; St. Dizier, Avery; Weigandt, Anna. \emph{Degrees of symmetric Grothendieck polynomials and Castelnuovo-Mumford regularity.} Proc. Amer. Math. Soc. 149 (2021), no. 4, 1405--1416.

\bibitem{Woo}
Ulfarsson, Henning; Woo, Alexander. \emph{Which Schubert varieties are local complete intersections?} Proc. Lond. Math. Soc. (3) 107 (2013), no. 5, 1004--1052.

\end{thebibliography}
\end{document}